\documentclass[12pt,reqno]{article}

\usepackage[usenames]{color}
\usepackage{amssymb}
\usepackage{amsmath}
\usepackage{amsthm}
\usepackage{amsfonts}
\usepackage{amscd}
\usepackage{graphicx}

\usepackage[colorlinks=true,
linkcolor=webgreen,
filecolor=webbrown,
citecolor=webgreen]{hyperref}

\definecolor{webgreen}{rgb}{0,.5,0}
\definecolor{webbrown}{rgb}{.6,0,0}

\usepackage{color}
\usepackage{fullpage}
\usepackage{float}

\usepackage{graphics}
\usepackage{latexsym}
\usepackage{epsf}
\usepackage{breakurl}
\usepackage{accents}

\begin{document}

\theoremstyle{plain}
\newtheorem{theorem}{Theorem}
\newtheorem{corollary}[theorem]{Corollary}
\newtheorem{lemma}{Lemma}
\newtheorem{example}{Example}

\begin{center}
	\vskip 1cm{\LARGE\bf New binomial Fibonacci sums \\
		\vskip .1in }
	
	\vskip .8cm
	
	{\large
		
		Kunle Adegoke \\
		Department of Physics and Engineering Physics, \\ Obafemi Awolowo University, Ile-Ife, Nigeria \\
		\href{mailto:adegoke00@gmail.com}{\tt adegoke00@gmail.com}
		
		\vskip 0.1 in
		
		Robert Frontczak\footnote{Statements and conclusions made in this article by R. Frontczak are entirely those of the author. They do not necessarily  reflect the views of LBBW.
			} \\
		Landesbank Baden-W\"urttemberg, Stuttgart,  Germany \\
		\href{mailto:robert.frontczak@lbbw.de}{\tt robert.frontczak@lbbw.de}
		
		\vskip 0.1 in
		
		Taras Goy  \\
		Faculty of Mathematics and Computer Science\\
		Vasyl Stefanyk Precarpathian National University, Ivano-Frankivsk, Ukrai\-ne\\
		\href{mailto:taras.goy@pnu.edu.ua}{\tt taras.goy@pnu.edu.ua}}
\end{center}

\vskip .2 in

\begin{abstract}
	\noindent We present some new linear, quadratic, cubic and quartic binomial Fibonacci, Lucas and Fibonacci--Lucas summation identities.
\end{abstract}

\section{Introduction}

Our goal is to derive, from elementary identities, some presumably new Fibonacci and Lucas identities including binomial coefficients. The research is a continuation of the recent works by Adegoke \cite{adegoke2,adegoke1} and Adegoke et~al.~\cite{adegoke3}.
The results are similar to those found in the classical articles of Carlitz and Ferns \cite{CarFer}, Hoggatt et al. \cite{hoggatt71}, Layman \cite{Lay}, Carlitz \cite{Car} and Long \cite{long90}. 

Recall that the Fibonacci numbers $F_j$ and the Lucas numbers $L_j$ are defined, for \text{$j\in\mathbb Z$}, through the recurrence relations
$$F_j = F_{j-1}+F_{j-2}, \quad j\ge 2,\quad F_0=0, \,\,F_1=1,
$$
$$L_j = L_{j-1}+L_{j-2},\quad j\ge 2, \quad L_0=2, \,\, L_1=1,
$$ 
with 
$$F_{-j} = (-1)^{j-1}F_j,\qquad L_{-j} = (-1)^j L_j.
$$

Throughout this paper, we denote the golden ratio by $\alpha=\frac{1+\sqrt 5}2$ and write $\beta=-\frac{1}{\alpha}$,
so that $\alpha\beta=-1$ and $\alpha+\beta=1$. Binet formulas for the Fibonacci and Lucas numbers are
\begin{equation}
F_j = \frac{{\alpha ^j - \beta ^j }}{{\alpha - \beta }},\quad L_j = \alpha ^j + \beta ^j, \quad j\in\mathbb Z. \label{binet}
\end{equation}

Koshy \cite{koshy} and Vajda \cite{vajda} have written excellent books dealing with Fibonacci and Lucas numbers.

Here are a couple of results to whet the readers appetite for reading on:
\begin{align*}
\sum_{k = 1}^n \binom{2n - 1}{2k - 1}F_{2k - 1}   
&= \begin{cases}
\frac12 F_{2n - 1} L_{n - 1} L_n, & \text{$n$ odd;}\\
\frac52 F_{2n - 1} F_{n - 1} F_n , & \text{$n$ even,}
\end{cases}\\
\sum_{k = 1}^n {\binom{2n - 1}{2k - 1}L_{2k - 1} } & = \frac12(L_{4n - 2} - L_{2n - 1}),\\
\sum_{k = 0}^n\binom{2n}{2k}F_{3k + r} F_{3k + s}  &= 2^{2n - 1}F_n F_{3n + r + s},\\
\sum_{k = 0}^{\left\lfloor {n}/{2} \right\rfloor} \binom {n}{2k} {2^{n-2k+1}} F_{2k + s} &=
\begin{cases}
(-1)^{s+1} F_{2n - s} + 5^{{n}/{2}}F_{n+s}, & \text{\rm $n$ even;}\\  (-1)^{s+1} F_{2n - s} + 5^{({n-1})/{2}}L_{n+s}, & \text{\rm $n$ odd.} \end{cases}\\
\sum_{k = 0}^{n} (- 1)^k \binom {2n}{2k} L_{j(2rk + s)} 
&	= 
\begin{cases}
\sqrt{5^n} F_{jr}^n L_{jrn + js} \cos \big(2n\arctan (\alpha ^{jr} )\big), & \text{\rm $jr$ odd;}\\
L_{jr}^n L_{jrn + js} \cos \big(2n\arctan (\alpha ^{jr} )\big), & \text{\rm $jr$ even, $n$ even;}\\
\sqrt5L_{jr}^n F_{jrn + js} \cos \big(2n\arctan (\alpha ^{jr} )\big), & \text{\rm $jr$ even, $n$ odd.} \end{cases}
\end{align*}

\section{Required identities}

\begin{lemma}[K.~Adegoke  \cite{adegoke2}]
	For real or complex $z$, let a given well-behaved 
	function $h(z)$ have in its domain the representation $h(z)=\sum\limits_{k=c_1}^{c_2}{g(k)z^{f(k)}}$, where $f(k)$ and $g(k)$ are given real sequences and $-\infty\leq c_1\leq c_2\leq \infty$. Let $j$ be an integer. Then
	\begin{align}
	\label{eq.y4g4no6}
	\sqrt 5 \sum_{k = c_1}^{c_2} {g(k)z^{f(k)} F_{jf(k)} } & = h(\alpha ^{j}z ) - h(\beta ^{j}z ),\\
	\label{eq.r2vyn7u}
	\sum_{k = c_1}^{c_2} {g(k)z^{f(k)} L_{jf(k)} } & = h(\alpha ^{j}z ) + h(\beta ^{j}z ).
	\end{align}
\end{lemma}
\begin{lemma}[S.~Vajda \cite{vajda}] \label{lem.crcmndb}
	For integers $r$ and $s$,
	\begin{align}
	&F_{r + s} + (-1)^sF_{r - s} = L_sF_r, \qquad\quad F_{r + s} - (-1)^sF_{r - s} = F_sL_r, \label{eq.t3h2005}\\
	&L_{r + s} + (-1)^sL_{r - s} = L_sL_r, \qquad\quad  L_{r + s} - (-1)^sL_{r - s} = 5F_sF_r. \label{eq.iix4fk6}
	\end{align}
\end{lemma}

If $u$ and $v$ are integers having the same parity, then identities \eqref{eq.t3h2005} and \eqref{eq.iix4fk6} can be put in the useful versions:
\begin{align*}
F_u & + ( - 1)^{\frac{u - v}2} F_v  = L_{\frac{u - v}{2}} F_{\frac{u + v}2},\qquad\quad
F_u  - ( - 1)^{\frac{u - v}2} F_v  = F_{\frac{u - v}2} L_{\frac{u + v}2},\\ 
L_u & + ( - 1)^{\frac{u - v}2} L_v  = L_{\frac{u - v}2} L_{\frac{u + v}2},\qquad\quad 
L_u  - ( - 1)^{\frac{u - v}2} L_v  = 5F_{\frac{u - v}2} F_{\frac{u + v}2},
\end{align*}
and
\begin{align*}
F_u  + F_v  &= 
\begin{cases}
L_{\frac{u - v}{2}} F_{\frac{u + v}{2}}, & \text{$\frac{u - v}{2}$ even;}\\
F_{\frac{u - v}{2}} L_{\frac{u + v}{2}}, & \text{$\frac{u - v}{2} $ odd,} \end{cases}\qquad\quad 
F_u  - F_v  =
\begin{cases}
L_{\frac{u - v}{2}} F_{\frac{u + v}{2}}, & \text{$\frac{u - v}{2}$ odd;}\\
F_{\frac{u - v}{2}} L_{\frac{u + v}{2}}, & \text{$\frac{u - v}{2}$ even,} \end{cases}\\
L_u  + L_v & = 
\begin{cases}
L_{\frac{u - v}{2}} L_{\frac{u + v}{2}}, & \text{$\frac{u - v}{2}$ even;}\\  5F_{\frac{u - v}{2}} F_{\frac{u + v}{2}}, & \text{$\frac{u - v}{2}$ odd,} \end{cases} \qquad\quad
L_u  - L_v  = 
\begin{cases}
L_{\frac{u - v}{2}} L_{\frac{u + v}{2}}, & \text{$\frac{u - v}{2}$ odd;}\\  5F_{\frac{u - v}{2}} F_{\frac{u + v}{2}}, & \text{$\frac{u - v}{2}$ even.} \end{cases}
\end{align*}
\begin{lemma} [K.~Adegoke  \cite{adegoke2}]
	\label{lem.dgcin1i}
	For $p$ and $q$ integers,
	\begin{align}\label{eq.z1yfyw6}
	1 + ( - 1)^p \alpha ^{2q}  
	&= \begin{cases}
	( - 1)^p \alpha ^q \sqrt 5 F_q , & \text{\rm $p-q$ odd;}\\
	( - 1)^p \alpha ^q L_q, & \text{\rm $p-q$ even,}
	\end{cases}\\
	\label{eq.xcamhlp}
	1 - ( - 1)^p \alpha ^{2q} & =
	\begin{cases}
	( - 1)^{p - 1} \alpha ^q L_q, & \text{\rm  $p-q$ odd;}\\
	( - 1)^{p - 1} \alpha ^q \sqrt 5 F_q , & \text{\rm  $p-q$ even.}
	\end{cases}
	\end{align}
\end{lemma}

The following formulas could be easily derived from the Binet formulas \eqref{binet}.
\begin{lemma} For $p$ and $q$ integers,
	\begin{align}
	&( - 1)^q  + \alpha ^{2q}  = \alpha ^q L_q, \qquad\quad ( - 1)^q  - \alpha ^{2q}  =  -\sqrt 5 \alpha ^q F_q , \label{eq.u9uuagc}\\
	&( - 1)^q  + \beta ^{2q}  = \beta ^q L_q, \qquad\quad ( - 1)^q  - \beta ^{2q}  =\sqrt 5 \beta ^q.\label{eq.syrjcay}
	\end{align}
\end{lemma}
\begin{lemma}[V.E.~Hoggatt, Jr. et al. \cite{hoggatt71}] \label{lem.ydalnfx}
	For $p$ and $q$ integers,
	\begin{align*}
	&L_{p + q}  - L_p \alpha ^q  =  -  \sqrt 5\beta ^p F_q,\qquad\quad L_{p + q}  - L_p \beta ^q  =  \sqrt 5\alpha ^p F_q ,\\
	&F_{p + q}  - F_p \alpha ^q  = \beta ^p F_q,\qquad \qquad\quad
	F_{p + q}  - F_p \beta ^q  = \alpha ^p F_q.
	\end{align*}
\end{lemma}
\begin{lemma}\label{lem.h02mtu1} We have
	\begin{gather}
	1 - \alpha = \beta ,\qquad 1 - \beta = \alpha ,\qquad 1 + \alpha = \alpha^2 ,\qquad 1 + \beta = \beta^2,
	\label{eq.uorwakj}\\
	1 - \alpha^3 = - 2\alpha ,\qquad 1 - \beta^3 = - 2\beta,\qquad 1 + \alpha^3 = 2\alpha^2 ,\qquad 1 + \beta^3 = 2\beta^2,
	\label{eq.vi8gjsp}\\
	1 - 2\alpha = - \sqrt 5 ,\qquad 1 - 2\beta = \sqrt 5 ,\qquad 1 + 2\alpha = \alpha^3 ,\qquad 1 + 2\beta = \beta^3,
	\label{eq.ss9cyu8}\\
	2 - \alpha = \beta^2 ,\qquad 2 - \beta = \alpha^2,\qquad 2 + \alpha = \alpha\sqrt 5 ,\qquad 2 + \beta = -\beta\sqrt 5,
	\label{eq.dirp56v}\\
		1 - \alpha^3\sqrt 5 = - 2\alpha^3 ,\quad\, 1 - \beta^3\sqrt 5 = 4\beta^2,\quad\, 1 + \alpha^3\sqrt 5 = 4\alpha^2, \quad\, 1 + \beta^3\sqrt 5 = -2\beta^3,
	\label{eq.ta648o1}\\
			\sqrt 5 - \alpha^3 = - 2,\qquad\sqrt 5 - \beta^3 = - 4\beta ,\qquad\sqrt 5 + \alpha^3 = 4\alpha ,\,\quad\sqrt 5 + \beta^3 = 2,
	\label{eq.ra1lsob}\\
		3 - \alpha^3 = 2\beta ,\qquad 3 - \beta^3 = 2\alpha,\qquad 3 + \alpha^3 = 2\alpha \sqrt 5 ,\qquad 3 + \beta^3 = - 2\beta \sqrt 5
	\label{eq.bt93a3y},\\
	1 - 3\alpha^3 = -2\alpha^2\sqrt 5 ,\qquad 1 - 3\beta^3 = 2\beta^2\sqrt 5,\qquad 
	1 + 3\alpha ^3 = 2\alpha^4 ,\qquad 1 + 3\beta^3 = 2\beta^4.
	\label{eq.lapedpo}
	\end{gather}
\end{lemma}
\begin{proof} Each identity is obtained by making appropriate substitutions for $p$ and $q$ in the identities given in Lemma \ref{lem.ydalnfx}.
\end{proof}

\section{Binomial summation identities, Part 1}

The first two key identities used frequently in this part of the paper are stated in the next fundamental lemma.
\begin{lemma}
	For integers $j$, $r$, $s$ and $n$ with $n$ non-negative, we have the following identities
	\begin{equation}
	\label{eq.upvt5hu}
	\begin{split}
	2\sqrt 5 \sum_{k = 0}^{\left\lfloor {n}/{2} \right\rfloor } \binom n{2k} &x^{n - 2k} z^{2k} F_{j(2rk + s)} \\
	&= \alpha ^{js}\left( (x + \alpha ^{jr} z)^n + (x - \alpha ^{jr} z)^n\right)
	- \beta ^{js}\left( (x + \beta ^{jr} z)^n  + (x - \beta ^{jr} z)^n\right),
	\end{split}
	\end{equation}
	\begin{equation} \label{eq.b7rf05c}
	\begin{split}
	2\sum_{k = 0}^{\left\lfloor {n}/{2} \right\rfloor } \binom n{2k}&x^{n - 2k} z^{2k} L_{j(2rk + s)} \\
	&= \alpha ^{js} \left((x + \alpha ^{jr} z)^n  + (x - \alpha ^{jr} z)^n\right)
	+ \beta ^{js} \left((x + \beta ^{jr} z)^n  + (x - \beta ^{jr} z)^n\right).
	\end{split}
	\end{equation}
\end{lemma}
\begin{proof} In the identity
	\begin{equation}\label{eq.b5o7ghm}
	h (z) = 2\sum_{k = 0}^{\left\lfloor {n}/{2} \right\rfloor } {\binom n{2k}x^{n-2k}z^{2rk + s} }  = z^s (x + z^r )^n  + z^s (x - z^r )^n
	\end{equation}
	identify $g(k)=2\binom n{2k}x^{n-2k}$, $f(k)=2rk+s$, $c_1=0$, $c_2=\left\lfloor {n}/2 \right\rfloor$, 
	and use these in \eqref{eq.y4g4no6},  \eqref{eq.r2vyn7u}.
\end{proof}
\begin{theorem}\label{thm.rz6m834}
	For non-negative integer $n$ and any integers $s$ and $j$,
	\begin{align}
	2\sum_{k = 0}^n {\binom{2n}{2k}F_{j(4k + s)} }& = \left( {L_j^{2n}  + 5^n F_j^{2n} } \right)F_{j(2n + s)},\label{eq.yqfik8x}\\
	2\sum_{k = 0}^n \binom{2n}{2k}L_{j(4k + s)}  & = \left( {L_j^{2n}  + 5^n F_j^{2n} } \right)L_{j(2n + s)},\label{eq.hrs065m}
	\end{align}
\vspace{-14 pt}
	\begin{align}
	2\sum_{k = 0}^{n - 1} \binom{2n-1}{2k}&F_{j(4k + s)}= ( - 1)^j\big(L_j^{2n - 1} F_{j(2n + s - 1)}  -  5^{n - 1} F_j^{2n - 1} L_{j(2n + s - 1)}\big),\label{eq.dv1m3af}\\
		2\sum_{k = 0}^{n - 1} \binom{2n-1}{2k} &L_{j(4k + s)}   =( - 1)^j\big( L_j^{2n - 1} L_{j(2n + s - 1)}  -  5^n F_j^{2n - 1} F_{j(2n + s - 1)}\big )\label{eq.blds4xe}. 
	\end{align}
\end{theorem}
\begin{proof} In \eqref{eq.upvt5hu}, set $x=(-1)^j$, $z=1$ and $r=2$ and use \eqref{eq.u9uuagc}, \eqref{eq.syrjcay} to obtain
	\[
	2\sqrt 5 \sum_{k = 0}^{\left\lfloor n/2 \right\rfloor } {\binom n{2k}F_{j(4k + s)} }= L_j^n \bigl( {\alpha ^{j(n + s)}  - \beta ^{j(n + s)} } \bigr)+ (-\sqrt 5 )^n F_j^n \bigl( {\alpha ^{j(n + s)}  - ( - 1)^n \beta ^{j(n + s)} } \bigr),
	\]
	from which \eqref{eq.yqfik8x} and \eqref{eq.dv1m3af} follow from the parity of $n$ and the Binet formulas.
	
	 The proof of \eqref{eq.hrs065m}, \eqref{eq.blds4xe} is similar; use $x=(-1)^j$, $z=1$ and $r=2$ in \eqref{eq.b7rf05c}.
\end{proof}
\begin{theorem} For non-negative integer $n$ and any integer $s$, we have 
	\begin{align}
	2\sum_{k = 0}^{\left\lfloor {n}/2 \right\rfloor } {\binom n{2k}F_{2k + s} }& = F_{2n + s} - (- 1)^s F_{n - s},\label{eq.ffgnfzm}\\  
	2\sum_{k = 0}^{\left\lfloor {n}/2 \right\rfloor } {\binom n{2k}L_{2k + s} } &= L_{2n + s} + (- 1)^s L_{n - s}. \label{eq.j0dzhsw}
	\end{align}
\end{theorem}
\begin{proof} Set $x=z=j=r=1$ in \eqref{eq.upvt5hu} and \eqref{eq.b7rf05c}, and use \eqref{eq.uorwakj}. This gives
	\begin{align*}
	2\sqrt 5 \sum_{k = 0}^{\left\lfloor {n}/{2} \right\rfloor} \binom {n}{2k}F_{2k + s} &= \alpha^{2n + s} - \beta^{2n + s}
	- (\alpha \beta )^s (\alpha^{n - s} - \beta^{n - s}),\\
	2\sum_{k = 0}^{\left\lfloor {n}/2 \right\rfloor} \binom {n}{2k}L_{2k + s} &= \alpha^{2n + s} + \beta^{2n + s}
	+ (\alpha \beta )^s (\alpha^{n - s}  + \beta^{n - s}),
	\end{align*}
	from which the stated identities follow immediately from the Binet formulas.
\end{proof}
\begin{corollary} For non-negative integer $n$ and any integer $s$,
	\begin{align*}
	2\sum_{k = 0}^n {\binom{2n}{2k}F_{2k + s} }  = &
	\begin{cases}
	L_{n + s} F_{3n}, & \text{\rm $n$ odd;}\\
	F_{n + s} L_{3n}, & \text{\rm  $n$ even,}
	\end{cases}\\
	2\sum_{k = 0}^n {\binom{2n}{2k}L_{2k + s} }  = &
	\begin{cases}
	5F_{n + s} F_{3n}, & \text{\rm  $n$ odd;}\\
	L_{n + s} L_{3n}, & \text{\rm  $n$ even.}
	\end{cases}
	\end{align*}
\end{corollary}
\begin{proof} Write $2n$ for $n$ in each of the identities~\eqref{eq.ffgnfzm} and \eqref{eq.j0dzhsw} and use Lemma~\ref{lem.crcmndb}.
\end{proof}
\begin{corollary} For positive integer $n$,
	\begin{align*}
	2\sum_{k = 0}^{n - 1} {\binom{2n - 1}{2k}F_{2k} } & = 
	\begin{cases}
	5F_{n - 1} F_{n} F_{2n - 1}, & \text{\rm \rm $n$ odd;}\\
	L_{n - 1} L_{n} F_{2n - 1}, & \text{\rm\rm $n$ even,}
	\end{cases}
	\\
	2\sum_{k = 0}^{n - 1} {\binom{2n - 1}{2k}L_{2k} } & = L_{4n - 2}  + L_{2n - 1}.
	\end{align*}
\end{corollary}
\begin{proof}
	Replace $n$ by $2n-1$ in each of the identities \eqref{eq.ffgnfzm} and \eqref{eq.j0dzhsw} and simplify.
\end{proof}
\begin{theorem}
	For non-negative integer $n$ and any integer $s$,
	\begin{align}
	\sum_{k = 0}^{\left\lfloor {n}/{2} \right\rfloor } \binom {n}{2k} F_{6k + s}&= 2^{n-1} \big(F_{2n+s} + (-1)^n F_{n+s}\big), \label{F_id_new1} \\
	\sum_{k = 0}^{\left\lfloor {n}/{2} \right\rfloor } \binom {n}{2k} L_{6k + s}&= 2^{n-1} \big(L_{2n+s} + (-1)^n L_{n+s}\big) \label{L_id_new1}.
	\end{align}
\end{theorem}
\begin{proof}
	Set $x=z=j=1$ and $r=3$ in \eqref{eq.upvt5hu} and \eqref{eq.b7rf05c}, use \eqref{eq.vi8gjsp} and simplify.
\end{proof}
\begin{corollary}
	For non-negative integer $n$ and any integer $s$, we have
	\begin{align}
	\sum_{k = 0}^n \binom {2n}{2k} F_{6k + s} & =  2^{2n-1} \begin{cases}
	L_{n} F_{3n+s}, & \text{\rm $n$ even;}\\
	F_{n} L_{3n+s}, & \text{\rm $n$ odd,}
	\end{cases}
	\nonumber\\
	\label{eq.rnor47h}
	\sum_{k = 0}^n \binom {2n}{2k} L_{6k + s} & = 2^{2n-1} 
	\begin{cases}
	L_{n} L_{3n+s}, & \text{\rm $n$ even;}\\ 5 F_{n} F_{3n+s}, & \text{\rm $n$ odd.} \end{cases}
	\end{align}
\end{corollary}
\begin{proof}
	Write $2n$ for $n$ in each of the identities \eqref{F_id_new1}, \eqref{L_id_new1} and use Lemma \ref{lem.crcmndb}.
\end{proof}
\begin{corollary}
	For positive integer $n$,
	\begin{align*}
	\sum_{k = 0}^{n - 1} \binom{2n - 1}{2k} F_{6k} &= 4^{n-1} 
	\begin{cases}
	5F_{n - 1} F_{n} F_{2n - 1}, & \text{\rm $n$ odd;}\\ L_{n - 1} L_{n} F_{2n - 1}, & \text{\rm $n$ even,} \end{cases}\\
	\sum_{k = 0}^{n - 1} \binom{2n - 1}{2k} L_{6k} &= 4^{n-1} (L_{4n - 2} - L_{2n - 1}).
	\end{align*}
\end{corollary}
\begin{proof} Replace $n$ by $2n-1$ in each of the identities \eqref{F_id_new1} and \eqref{L_id_new1} and simplify.
\end{proof}
\begin{theorem} For non-negative integer $n$ and any integer $s$,
	\begin{align}
	\sum_{k = 0}^{\left\lfloor {n}/{2} \right\rfloor } {\binom n{2k}5^k F_{6k + s} } &= 2^{n - 1}\big(2^n F_{2n + s} + (- 1)^n F_{3n + s}\big),\nonumber\\
	\sum_{k = 0}^{\left\lfloor {n}/{2} \right\rfloor } {\binom n{2k}5^k L_{6k + s} } &= 2^{n - 1} \big(2^n L_{2n + s} + (- 1)^n L_{3n + s}\big)\label{eq.a02pcu8}.
	\end{align}
\end{theorem}
\begin{proof} Setting $x=j=1$, $z=\sqrt 5$ and $r=3$ in \eqref{eq.upvt5hu} and \eqref{eq.b7rf05c} while making use of the identities \eqref{eq.ta648o1} gives
	\begin{align*}
	2\sqrt 5 \sum_{k = 0}^{\left\lfloor {n}/{2} \right\rfloor } {\binom {n}{2k}5^k F_{6k + s} }  &= 4^n \big(\alpha ^{2n + s}  - \beta ^{2n + s} \big) + ( - 2)^n \big(\alpha ^{3n + s}  - \beta ^{3n + s} \big),\\
	2\sum_{k = 0}^{\left\lfloor {n}/{2} \right\rfloor } {\binom {n}{2k}5^k L_{6k + s} }  7&= 4^n \big(\alpha ^{2n + s}  + \beta ^{2n + s} \big) + ( - 2)^n\big (\alpha ^{3n + s}  + \beta ^{3n + s} \big),
	\end{align*}
	from which the stated identities follow.
\end{proof}
\begin{corollary}
	For non-negative integer $n$ and any integer $s$, we have
	\begin{align*}
	\sum_{k = 0}^n \binom {2n}{2k} 5^k F_{6k + s}& = 2^{2n-1}\bigl(4^{n} F_{4n+s} + F_{6n+s}\bigr),\\
	\sum_{k = 0}^n \binom {2n}{2k} 5^k L_{6k + s} &= 2^{2n-1}\bigl( 4^{n}L_{4n+s} + L_{6n+s}\bigr).
	\end{align*}
\end{corollary}
\begin{theorem} For non-negative integer $n$ and any integer $s$,
	\begin{align}
	2\sum_{k = 0}^n {\binom{2n}{2k} \frac{F_{6k + s}}{5^k} }&  =\left(\frac{4}{5}\right)^{n}\bigl(4^{n} F_{2n + s}  + F_s\bigr), \label{eq.d68cj51}\\
	2\sum_{k = 0}^n {\binom{2n}{2k} \frac{L_{6k + s}}{5^k} } & = \left(\frac{4}{5}\right)^{n} \bigl( 4^n L_{2n + s}  + L_s\bigr), \label{eq.lrs2k39}\\
	8\sum_{k = 0}^n {\binom{2n - 1}{2k}\frac{F_{6k + s}}{5^k} } & = \left(\frac{4}{5}\right)^{n} \bigl(4^{n} L_{2n - 1 + s}  -2 L_s\bigr),\label{eq.fbeis5t}\\
	2\sum_{k = 0}^n {\binom{2n - 1}{2k}\frac{ L_{6k + s}}{5^k} }&  =\left(\frac{4}{5}\right)^{n-1} \bigl(4^{n}F_{2n - 1 + s}  - 2F_s\bigr)\label{eq.x1whoqn}.
	\end{align}
\end{theorem}
\begin{proof}
	Setting $z=j=1$, $x=\sqrt 5$ and  $r=3$ in \eqref{eq.upvt5hu} and \eqref{eq.b7rf05c} while making use of the identities \eqref{eq.ra1lsob} gives
	\begin{align}
	2\sum_{k = 0}^{\left\lfloor {n}/{2} \right\rfloor}\binom{n}{2k}{\frac{F_{6k + s}}{5^{k-\frac{n+1}{2}}}} &= 4^n \big(\alpha ^{n + s}  - ( - 1)^n \beta ^{n + s} \big) + ( - 2)^n \big(\alpha ^s  - ( - 1)^n \beta ^s \big),\label{eq.a2k96nt}\\
	2\sum_{k = 0}^{\left\lfloor {n}/{2} \right\rfloor } \binom{n}{2k}  \frac{L_{6k + s}}{5^{k-\frac{n}{2}}} & = 4^n \big(\alpha ^{n + s}  + ( - 1)^n \beta ^{n + s}\big) + ( - 2)^n \big(\alpha ^s  + ( - 1)^n \beta ^s\big )\label{eq.qm7xjr9}.
	\end{align}
	Writing $2n$ for $n$ in~\eqref{eq.a2k96nt} and~\eqref{eq.qm7xjr9} produces identities~\eqref{eq.d68cj51}, \eqref{eq.lrs2k39} while writing $2n - 1$ for $n$ yields identities~\eqref{eq.fbeis5t} and~\eqref{eq.x1whoqn}.
\end{proof}
\begin{theorem}\label{thm.gybpu6k}
	For non-negative integer $n$ and any integer $s$,
	\begin{align}
	2 \sum_{k = 0}^{\left\lfloor {n}/{2} \right\rfloor} \binom {n}{2k} 4^k F_{2k + s}& =
	\begin{cases}
	F_{3n + s} + 5^{{n}/{2}}F_s, & \text{\rm $n$ even;}\\ 
	F_{3n + s} - 5^{(n-1)/{2}}L_s, & \text{\rm $n$ odd,} \end{cases}
	\label{eq.bhdgtf0}\\
	2 \sum_{k = 0}^{\left\lfloor {n}/{2}  \right\rfloor} \binom {n}{2k} 4^k L_{2k + s}& =
	\begin{cases}
	L_{3n + s} + 5^{{n}/{2}}L_s, & \text{\rm $n$ even;}\\
	L_{3n + s} - 5^{(n+1)/2}F_s, & \text{\rm $n$ odd.} \end{cases}
	\end{align}
\end{theorem}
\begin{proof} Set $x=j=r=1$ and $z=2$ in \eqref{eq.upvt5hu},  \eqref{eq.b7rf05c} and make use of \eqref{eq.ss9cyu8}. The calculations are straightforward and omitted.
\end{proof}
\begin{theorem}
	For non-negative integer $n$ and any integer $s$,
	\begin{align*}
	\sum_{k = 0}^{\left\lfloor {n}/{2} \right\rfloor} \binom {n}{2k} \frac{F_{2k + s}}{4^k} &=\frac{1}{2^{n+1}}
	\begin{cases}
	(-1)^{s+1} F_{2n - s} + 5^{{n}/{2}}F_{n+s}, & \text{\rm $n$ even;}\\  (-1)^{s+1} F_{2n - s} + 5^{(n-1)/2}L_{n+s}, & \text{\rm $n$ odd,} \end{cases}
	\\
	\sum_{k = 0}^{\left\lfloor {n}/{2} \right\rfloor} \binom {n}{2k} \frac{L_{2k + s}}{4^{k}}&=\frac{1}{2^{n+1}}
	\begin{cases}
	(-1)^s L_{2n - s} + 5^{{n}/{2}}L_{n+s}, & \text{\rm $n$ even;}\\(-1)^s L_{2n - s} + 5^{(n+1)/2}F_{n+s}, & \text{\rm $n$ odd.} \end{cases}
	\end{align*}
\end{theorem}
\begin{proof} Set $x=2$, $z=-1$ and $r=j=1$ in \eqref{eq.upvt5hu} and \eqref{eq.b7rf05c} and make use of \eqref{eq.dirp56v}. The calculations are straightforward and omitted.
\end{proof}
\begin{theorem}
	For non-negative integer $n$ and any integer $s$,
	\begin{align*}
	2\sum_{k = 0}^{\left\lfloor {n}/{2} \right\rfloor} \binom {n}{2k} \frac{F_{6k + s}}{9^k}& = \left(\frac23\right)^{n}
	\begin{cases}
	(-1)^{s+1} F_{n - s} + 5^{{n}/{2}}F_{n+s}, & \text{\rm $n$ even;}\\ (-1)^{s+1} F_{n - s} + 5^{(n-1)/{2}}L_{n+s}, & \text{\rm $n$ odd,} \end{cases}
	\\
	2\sum_{ k = 0}^{\left\lfloor {n}/{2} \right\rfloor} \binom {n}{2k} \frac{L_{6k + s}}{9^k}& =\left(\frac23\right)^{n} 
	\begin{cases}
	(-1)^s L_{n - s} + 5^{{n}/{2}}L_{n+s}, & \text{\rm $n$ even;}\\ (-1)^s L_{n - s} + 5^{(n+1)/{2}}F_{n+s}, & \text{\rm $n$ odd.} \end{cases}
	\end{align*}
\end{theorem}
\begin{proof} 	Set $x=r=3$, $z=-1$, $j=1$ in \eqref{eq.upvt5hu} and \eqref{eq.b7rf05c} and make use of \eqref{eq.bt93a3y}. The calculations are omitted.
\end{proof}
\begin{theorem} For non-negative integer $n$ and any integer $s$,
	\begin{align*}
	\sum_{k = 0}^{\left\lfloor n/2 \right\rfloor} \binom {n}{2k} 9^{k} F_{6k + s}& = 2^{n-1}
	\begin{cases}
	F_{4n + s} + 5^{{n}/{2}} F_{2n+s}, & \text{\rm $n$ even;}\\ F_{4n + s} - 5^{(n-1)/2}L_{2n+s}, & \text{\rm $n$ odd,} \end{cases}
	\\
	\sum_{k = 0}^{\left\lfloor {n}/{2} \right\rfloor} \binom {n}{2k} 9^{k} L_{6k + s} & = 2^{n-1}
	\begin{cases}
	L_{4n + s} + 5^{{n}/{2}}L_{2n+s}, & \text{\rm $n$ even;}\\  L_{4n + s} - 5^{(n+1)/{2}}F_{2n+s}, & \text{\rm $n$ odd.} \end{cases}
	\end{align*}
\end{theorem}
\begin{proof} Set $x=j=1$ and $z=r=3$ in \eqref{eq.upvt5hu} and \eqref{eq.b7rf05c} and make use of \eqref{eq.lapedpo}. The calculations are omitted.
\end{proof}

\section{Binomial summation  identities, Part 2}

This section is based on the following fundamental lemma.
\begin{lemma}
	For integers $j$, $r$, $s$ and $n$ with $n$ non-negative, we have 
	\begin{equation}
	\label{eq.uglu0yf}
	\begin{split}
	2\sqrt 5 \sum_{k = 1}^{\left\lceil {n}/{2} \right\rceil} &\binom {n}{2k - 1} x^{n - 2k + 1} z^{2k - 1} F_{j(2rk + s)} \\
	&= \alpha ^{j(r + s)} \big((x + \alpha ^{jr} z)^n  - (x - \alpha ^{jr} z)^n\big)
	- \beta ^{j(r + s)}\big( (x + \beta ^{jr} z)^n  - (x - \beta ^{jr} z)^n\big),
	\end{split}
	\end{equation}
	\begin{equation}
	\label{eq.e120i21}
	\begin{split}
	2\sum_{k = 1}^{\left\lceil {n}/{2} \right\rceil} &\binom {n}{2k - 1}x^{n - 2k + 1} z^{2k - 1} L_{j(2rk + s)} \\
	&= \alpha ^{j(r + s)} \big((x + \alpha ^{jr} z)^n  - (x - \alpha ^{jr} z)^n\big)
	+ \beta ^{j(r + s)} \big((x + \beta ^{jr} z)^n  -(x - \beta ^{jr} z)^n\big).
	\end{split}
	\end{equation}
\end{lemma}
\begin{proof} In the identity
	$$
	h (z) = 2\sum_{k = 1}^{\left\lceil {n}/{2} \right\rceil } {\binom n{2k - 1}x^{n - 2k}z^{2rk + s} }  = z^{r + s} (x + z^r )^n  - z^{r + s} (x - z^r )^n,
	$$
	identify $g(k)=2\binom n{2k - 1}x^{n-2k},$ $f(k)=2rk+s$, $c_1=1$, $c_2=\left\lceil {n/2} \right\rceil$, and use these in \eqref{eq.y4g4no6} and \eqref{eq.r2vyn7u}.
\end{proof}
\begin{theorem} For non-negative integer $n$ and integer $s$,
	\begin{align}
	2\sum_{k = 1}^{\left\lceil {n}/{2} \right\rceil } {\binom n{2k - 1}F_{2k + s} } & = F_{2n + s + 1}  - ( - 1)^s F_{n - s - 1},\label{th10}\\
	2\sum_{k = 1}^{\left\lceil {n}/{2} \right\rceil } {\binom n{2k - 1}L_{2k + s} } & = L_{2n + s + 1}  + ( - 1)^s L_{n - s - 1}.\nonumber
	\end{align}
\end{theorem}
\begin{proof} Set $x=z=j=r=1$ in \eqref{eq.uglu0yf} and \eqref{eq.e120i21} to obtain
	\begin{align*}
	2\sqrt 5 \sum_{k = 1}^{\left\lceil {n}/{2} \right\rceil } \binom n{2k - 1}F_{2k + s}  & = \alpha ^{2n + s + 1}  - \beta ^{2n + s + 1}  + (\alpha \beta )^{s + 1} (\alpha ^{n - s - 1}  - \beta ^{n - s - 1} ),\\
	2\sum_{k = 1}^{\left\lceil {n}/{2} \right\rceil } \binom n{2k - 1}L_{2k + s} &  = \alpha ^{2n + s + 1}  + \beta ^{2n + s + 1} - (\alpha \beta )^{s + 1} (\alpha ^{n - s - 1}  + \beta ^{n - s - 1} ),
	\end{align*}
	and hence the stated identities.
\end{proof}
\begin{corollary} For non-negative integer $n$ and integer $s$,
	\begin{align}
	2\sum_{k = 1}^n {\binom {2n}{2k - 1}F_{2k + s} } & = 
	\begin{cases}
	L_{n + s + 1} F_{3n}, & \text{\rm $n$ even;}\\ F_{n + s + 1} L_{3n}, & \text{\rm  $n$ odd,} \end{cases}\label{coro6_1}
	\\
	2\sum_{k = 1}^n {\binom {2n}{2k - 1}L_{2k + s} } & = 
	\begin{cases}
	5F_{n + s + 1} F_{3n}, & \text{\rm  $n$ even;}\\ L_{n + s + 1} L_{3n}, & \text{\rm $n$ odd.}\nonumber \end{cases}
	\end{align}
\end{corollary}
\begin{proof} We prove \eqref{coro6_1}. From \eqref{th10}, using \eqref{eq.t3h2005} we have 
	\begin{equation*}
	\begin{split}
	2\sum_{k = 1}^n \binom {2n}{2k - 1}F_{2k + s}   &= F_{4n + s + 1}  - ( - 1)^s F_{2n - s - 1}\\
	&= F_{3n + (n + s + 1)}  + ( - 1)^{s + 1} F_{3n - (n + s + 1)}\\
	&= 
	\begin{cases}
	L_{n + s + 1} F_{3n}, & \text{\rm $n$ even;}\\ F_{n + s + 1} L_{3n}, & \text{\rm  $n$ odd.} \end{cases}
	\end{split}
	\end{equation*}
\end{proof}
\begin{corollary} For positive integer $n$,
	\begin{align*}
	2\sum_{k = 1}^n {\binom{2n - 1}{2k - 1}F_{2k - 1} } & =
	\begin{cases}
	F_{2n - 1} L_{n - 1} L_n , & \text{\rm$n$ odd;}\\  5F_{2n - 1} F_{n - 1} F_n, & \text{\rm$n$ even.} \end{cases}\\
	2\sum_{k = 1}^n {\binom{2n - 1}{2k - 1}L_{2k - 1} }&  = L_{4n - 2}  - L_{2n - 1}.
	\end{align*}
\end{corollary}
\begin{theorem} For non-negative integer $n$ and any integer $s$,
	\begin{align*}
	\sum_{k = 1}^{\left\lceil {n}/{2} \right\rceil } \binom {n}{2k-1} F_{6k + s} &= 2^{n-1} \big(F_{2n+3+s} - (-1)^n F_{n+3+s}\big), \\ 
	\sum_{k = 1}^{\left\lceil {n}/{2}  \right\rceil } \binom {n}{2k-1} L_{6k + s}& = 2^{n-1}\big(L_{2n+3+s} - (-1)^n L_{n+3+s}\big).
	\end{align*}
\end{theorem}
\begin{proof} 	Set $x=z=j=1$ and $r=3$ in \eqref{eq.uglu0yf} and \eqref{eq.e120i21}, use~\eqref{eq.vi8gjsp} and simplify.
\end{proof}
\begin{corollary} For positive integer $n$ and any integer $s$,
	\begin{align*}
	\sum_{k = 1}^{n} \binom {2n-1}{2k-1} F_{6k + s} &= 4^{n-1} \big(F_{4n+1+s} + F_{2n+2+s}\big), \\
	\sum_{k = 1}^{n} \binom {2n-1}{2k-1} L_{6k + s}& = 4^{n-1}\big (L_{4n+1+s} + L_{2n+2+s}\big).
	\end{align*}
\end{corollary}
\begin{theorem} For non-negative integer $n$ and any integer $s$,
	\begin{align*}
	\sum_{k = 1}^{\left\lceil {n}/{2} \right\rceil} \binom {n}{2k-1} 4^k F_{2k + s} &=
	\begin{cases}
	F_{3n + 1 + s} - 5^{{n}/{2}}F_{s+1}, & \text{\rm$n$ even;}\\ F_{3n + 1 + s} + 5^{(n-1)/2}L_{s+1}, & \text{\rm $n$ odd,
	} \end{cases}
	\\
	\sum_{k = 1}^{\left\lceil {n}/{2} \right\rceil} \binom {n}{2k-1} 4^k L_{2k + s} &=
	\begin{cases}
	L_{3n + 1 + s} - 5^{{n}/{2}} L_{s+1}, & \text{\rm$n$ even;}\\  L_{3n + 1 + s} + 5^{(n+1)/2}F_{s+1}, & \text{\rm$n$ odd.} \end{cases}
	\end{align*}
\end{theorem}
\begin{proof} Set $x=j=r=1$ and $z=2$ in \eqref{eq.uglu0yf} and \eqref{eq.e120i21}, use~\eqref{eq.ss9cyu8} and simplify.
\end{proof}
\begin{theorem} For non-negative integer $n$ and any integer $s$,
	\begin{align*}
	\sum_{k = 1}^{\left\lceil {n}/{2} \right\rceil} \binom {n}{2k-1}  \frac{F_{2k + s}}{4^k} &=
	\frac{1}{2^{n+2}}\begin{cases}
	(-1)^{s+1} F_{2n - 1 - s} + 5^{{n}/{2}}F_{n+s+1}, & \text{\rm$n$ even;}\\ (-1)^{s+1} F_{2n - 1 - s} + 5^{(n-1)/2}L_{n+s+1}, & \text{\rm$n$ odd,} \end{cases}
	\\
	\sum_{k = 1}^{\left\lceil {n}/{2} \right\rceil} \binom {n}{2k-1}  \frac{L_{2k + s}}{4^k} &=
	\frac{1}{2^{n+2}}\begin{cases}
	(-1)^s L_{2n - 1 - s} + 5^{{n}/{2}}L_{n+s+1}, & \text{\rm$n$ even;}\\ (-1)^s L_{2n - 1 - s} + 5^{(n+1)/{2}}F_{n+s+1}, & \text{\rm$n$ odd.} \end{cases}
	\end{align*}
\end{theorem}
\begin{proof} Set $x=2$, $z=-1$, $j=r=1$ in \eqref{eq.uglu0yf} and \eqref{eq.e120i21}, use~\eqref{eq.dirp56v} and simplify.
\end{proof}
\begin{theorem} For non-negative integer $n$ and any integer $s$,
	\begin{align*}
	6\sum_{k = 1}^{\left\lceil {n}/{2} \right\rceil} \binom {n}{2k-1} \frac{F_{6k + s}}{9^k} &= \left(\frac{2}{3}\right)^n
	\begin{cases}
	(-1)^{s+1} F_{n - 3 - s} + 5^{{n}/{2}}F_{n+s+3}, & \text{\rm$n$ even};\\(-1)^{s+1} F_{n - 3 - s} + 5^{(n-1)/{2}}L_{n+s+3}, & \text{\rm$n$ odd,} \end{cases}
	\\
	6\sum_{k = 1}^{\left\lceil {n}/{2} \right\rceil}\binom {n}{2k-1} \frac{L_{6k + s}}{9^k} &= \left(\frac{2}{3}\right)^n
	\begin{cases}
	(-1)^s L_{n - 3 - s} + 5^{{n}/{2}}L_{n+s+3}, & \text{\rm$n$ even;}\\ (-1)^s L_{n - 3 - s} + 5^{({n+1})/{2}}F_{n+s+3}, & \text{\rm$n$ odd.} \end{cases}
	\end{align*}
\end{theorem}
\begin{proof} Set $x=r=3$, $z=-1$, $j=1$ in \eqref{eq.uglu0yf} and \eqref{eq.e120i21}, use~\eqref{eq.bt93a3y} and simplify.
\end{proof}
\begin{theorem} For non-negative integer $n$ and any integer $s$,
	\begin{align*}
	\sum_{k = 1}^{\left\lceil {n}/{2} \right\rceil} \binom {n}{2k-1} 9^{k} F_{6k + s} &= 2^{n-1}
	\begin{cases}
	F_{4n + 3 + s} - 5^{{n}/{2}}F_{2n+s+3}, & \text{\rm$n$ even;}\\  F_{4n + 3 + s} + 5^{({n-1})/{2}}L_{2n+s+3}, & \text{\rm$n$ odd.} \end{cases}
	\\
	\sum_{k = 1}^{\left\lceil {n}/{2} \right\rceil} \binom {n}{2k-1} 9^{k} L_{6k + s}& = {2^{n-1}}
	\begin{cases}
	L_{4n + 3 + s} - 5^{{n}/{2}}L_{2n+s+3}, & \text{\rm$n$ even;}\\ L_{4n + 3 + s} + 5^{({n+1})/{2}}F_{2n+s+3}, & \text{\rm$n$ odd.} \end{cases}
	\end{align*}
\end{theorem}
\begin{proof} Set $x=j=1$ and $z=r=3$ in \eqref{eq.uglu0yf} and \eqref{eq.e120i21}, use \eqref{eq.lapedpo} and simplify.
\end{proof}

\section{Binomial summation  identities, Part 3}

In this section we introduce the following results.
\begin{lemma} For integers $j$, $r$, $s$, $n$ with $n$ non-negative, we have the identi\-ties
	\begin{equation}
	\label{eq.ag3ooim}
	\begin{split}
	\sqrt 5 \sum_{k = 0}^{\left\lfloor {n}/{2} \right\rfloor } (- 1)^k \binom {n}{2k}&x^{n - 2k} z^{2k} F_{j(2rk + s)} \\
	& = \alpha^{js} \sqrt{(x^2  + \alpha^{2jr} z^2 )^{n}} \cos \left( {n\arctan \Bigl( {\frac{{\alpha^{jr} z}}{x}} \Bigr)} \right)\\
	&\quad - \beta^{js} \sqrt{(x^2 + \beta^{2jr} z^2 )^{n}} \cos \left( {n\arctan \Bigl( {\frac{{\beta^{jr} z}}{x}} \Bigr)} \right),
	\end{split}
	\end{equation}
	\begin{equation}
	\label{eq.g0km87g}
	\begin{split}
	\sum_{k = 0}^{\left\lfloor {n}/{2} \right\rfloor } (- 1)^k \binom {n}{2k}&x^{n - 2k} z^{2k} L_{j(2rk + s)} \\
	& = \alpha^{js} \sqrt{(x^2  + \alpha^{2jr} z^2 )^{n}} \cos \left( {n\arctan \Bigl( {\frac{{\alpha^{jr} z}}{x}} \Bigr)} \right)\\
	&\quad + \beta^{js} \sqrt{(x^2  + \beta^{2jr} z^2 )^{n}} \cos \left( {n\arctan \Bigl( {\frac{{\beta^{jr} z}}{x}} \Bigr)} \right),
	\end{split}
	\end{equation}
	\begin{equation*}
	\label{eq.dtvtoob}
	\begin{split}
	\sqrt 5 \sum_{k = 1}^{\left\lceil {n}/{2} \right\rceil } (- 1)^{k - 1}&\binom {n}{2k - 1}x^{n - 2k + 1} z^{2k - 1} F_{j(2rk + s)} \\
	&= \alpha^{js} \sqrt{(x^2 + \alpha^{2jr} z^2 )^{n}} \sin \left( {n\arctan \Bigl( {\frac{{\alpha^{jr} z}}{x}} \Bigr)} \right)\\
	&\quad- \beta^{js} \sqrt{(x^2 + \beta^{2jr} z^2 )^{n}}\sin \left( {n\arctan \Bigl( {\frac{{\beta^{jr} z}}{x}} \Bigr)} \right),
	\end{split}
	\end{equation*}
	\begin{equation*}
	\begin{split}
	\sum_{k = 1}^{\left\lceil {n}/{2} \right\rceil } ( - 1)^{k - 1} &\binom {n}{2k - 1}x^{n - 2k + 1} z^{2k - 1} L_{j(2rk + s)} \\
	&= \alpha^{js} \sqrt{(x^2 + \alpha^{2jr} z^2 )^{n}} \sin \left( {n\arctan \Bigl( {\frac{{\alpha^{jr} z}}{x}} \Bigr)} \right)\\
	&\quad + \beta^{js} \sqrt{(x^2 + \beta^{2jr} z^2 )^{n}} \sin \left( {n\arctan \Bigl( {\frac{{\beta^{jr} z}}{x}} \Bigr)} \right).
	\end{split}
	\end{equation*}
\end{lemma}
\begin{proof} In the identity
	\[
	z^s \sqrt{(x^2 + z^{2r} )^{n}} \cos \left(n\arctan \Bigl(\frac{z^r}{x}\Bigr)\right) = \sum_{k = 0}^{\left\lfloor {n}/{2} \right\rfloor }
	(- 1)^k \binom {n}{2k}x^{n - 2k} z^{2rk + s},
	\]
	identify $h(z) = z^s \sqrt{(x^2 + z^{2r} )^{n}} \cos \big(n\arctan (\frac{z^r}{x})\big)$, $g(k) = (- 1)^k \binom {n}{2k}x^{n - 2k}$, $f(k) = 2rk + s$, $c_1  = 0$, $c_2  = \left\lfloor {n/2} \right\rfloor$, and use these in \eqref{eq.y4g4no6} and \eqref{eq.r2vyn7u}. This proves \eqref{eq.ag3ooim} and \eqref{eq.g0km87g}.
	
	For the other two identities use
	\[
	z^{s+r} \sqrt{(x^2  + z^{2r} )^{n}} \sin \left(n\arctan \Bigl(\frac{z^r}{x}\Bigr)\right) = \sum_{k = 1}^{\left\lceil {n}/{2} \right\rceil } {( - 1)^{k - 1} \binom n{2k - 1}x^{n - 2k + 1} z^{2rk + s} },
	\]
	and identify $h(z) = z^{s+r} \sqrt{(x^2 + z^{2r} )^{n}} \sin\big(n\arctan (\frac{z^r}{x})\big)$, $f(k) = 2rk + s$, $c_1  = 1$, $c_2  = \left\lceil {n/2} \right\rceil$ and $g(k) = (- 1)^{k-1} \binom n{2k - 1}x^{n - 2k + 1}$.
\end{proof}

We give an example. The next lemma proves useful.
\begin{lemma}
	For integers $r$ and $n$,
	\begin{align*}
	\cos \big(2n\arctan (\alpha ^r )\big)& = (- 1)^n \cos \big(2n\arctan (\beta ^r )\big),\\ 
	\cos \big((2n - 1)\arctan (\alpha ^r )\big)& = (- 1)^{n + r + 1} \sin\big ((2n - 1)\arctan (\beta ^r )\big),\\
	\sin\big (2n\arctan (\alpha ^r )\big) &= (- 1)^{n + r + 1} \sin \big(2n\arctan (\beta ^r )\big),\\
	\sin \big((2n - 1)\arctan (\alpha ^r )\big)& = (- 1)^{n - 1} \cos \big((2n - 1)\arctan (\beta ^r )\big).
	\end{align*}
\end{lemma}
\begin{proof}
	On account of identities
	$\arctan x + \arctan(1/x) = \pi/2$, if $x>0$, and $\alpha^r=(-1)^r\beta^{-r}$,
	we have
	$$
	\arctan(\alpha^r) = \frac\pi2 - (-1)^r\arctan (\beta^r), 
	$$
	from which the identities follow upon applying addition formulas for trigonometric functions.
\end{proof}
\begin{theorem} For non-negative integer $n$ and integers $j$, $r$, $s$,
	\begin{align}
	\sum_{k = 0}^{n} (- 1)^k \binom {2n}{2k} F_{j(2rk + s)}
&	= 
	\begin{cases}
	\sqrt{5^n} F_{jr}^n F_{jrn + js} \cos \big(2n\arctan (\alpha ^{jr})\big), & \text{\rm $jr$ odd;}\\ 
	L_{jr}^n F_{jrn + js} \cos \big(2n\arctan (\alpha ^{jr} )\big), & \text{\rm $jr$ even, $n$ even;}\\ 
	\frac{1}{\sqrt5}L_{jr}^n L_{jrn + js} \cos\big (2n\arctan (\alpha ^{jr} )\big), & \text{\rm $jr$ even, $n$ odd,} \end{cases}\nonumber\\
	\label{eq.vfsdrta}
	\sum_{k = 0}^{n} (- 1)^k \binom {2n}{2k} L_{j(2rk + s)} 
&	= 
	\begin{cases}
	\sqrt{5^n} F_{jr}^n L_{jrn + js} \cos \big(2n\arctan (\alpha ^{jr} )\big), & \text{\rm $jr$ odd;}\\
	L_{jr}^n L_{jrn + js} \cos \big(2n\arctan (\alpha ^{jr} )\big), & \text{\rm $jr$ even, $n$ even;}\\
	\sqrt5L_{jr}^n F_{jrn + js} \cos \big(2n\arctan (\alpha ^{jr} )\big), & \text{\rm $jr$ even, $n$ odd.} \end{cases}
	\end{align}
\end{theorem}
\begin{proof} The choice $x=z=1$ in \eqref{eq.ag3ooim} and \eqref{eq.g0km87g}, noting also identities~\eqref{eq.u9uuagc} and \eqref{eq.syrjcay}, with $q=jr$, $jr$ odd, gives
	\begin{align*}
	\sum_{k = 0}^{n} (- 1)^k \binom {2n}{2k}F_{j(2rk + s)} &= 5^{(n-1)/2} F_{jr}^n \big(\alpha^{jrn + js} - \beta^{jrn + js}\big)\cos \big(2n\arctan (\alpha ^{jr} )\big),\\
	\sum_{k = 0}^{n} (- 1)^k \binom {2n}{2k}L_{j(2rk + s)} &= 5^{n/2} F_{jr}^n \big(\alpha^{jrn + js} + \beta ^{jrn + js}\big)\cos \big(2n\arctan (\alpha ^{jr} )\big),
	\end{align*}
	from which the stated identities for $jr$ odd follow. 
	
	Similarly, $x=z=1$ in \eqref{eq.ag3ooim} and \eqref{eq.g0km87g}, with $jr$ even, gives
	\begin{align*}
	\sum_{k = 0}^{n} (- 1)^k \binom {2n}{2k}F_{j(2rk + s)} &= \frac{L_{jr}^n}{\sqrt 5 } \big(\alpha^{jrn + js} - (- 1)^n \beta^{jrn + js}\big )\cos \big(2n\arctan (\alpha ^{jr} )\big),\\
	\sum_{k = 0}^{n} (- 1)^k \binom {2n}{2k}L_{j(2rk + s)}& = L_{jr}^n \big(\alpha^{jrn + js} + (- 1)^n \beta^{jrn + js} \big)\cos \big(2n\arctan (\alpha ^{jr} )\big),
	\end{align*}
	and hence the stated identities for $jr$ even.
\end{proof}
So far in this paper we have been concerned with identities that are linear in the Fibonacci and Lucas numbers. In the remaining three sections we will present results containing higher order binomial Fibonacci identities.

\section{Quadratic binomial summation  identities}

Here we will derive a pair of binomial identities involving products of Fibonacci and Lucas numbers. We require the results stated in the next two lemmas.
\begin{lemma}\label{lem.lk77qra}
	If $k$, $r$ and $s$ are integers, then
	\begin{align}
	5F_{k + r} F_{k + s}  &= L_{2k + r + s}  - ( - 1)^{k + s} L_{r - s},\label{eq.f5dydlo}\\ 
	L_{k + r} F_{k + s}  &= F_{2k + r + s}  - ( - 1)^{k + s} F_{r - s},\label{eq.eh9zyvi}\\ 
	L_{k + r} L_{k + s} & = L_{2k + r + s}  + ( - 1)^{k + s} L_{r - s}\label{eq.z9rhta5}.
	\end{align}
\end{lemma} \begin{proof} These are variations on Vajda \cite[Identities (15b), (17a), (17b)]{vajda}.
\end{proof}
\begin{lemma}
	If $n$ is a positive integer, then
	\begin{align}\label{eq.wj9bwlz}
	\sum_{k = 0}^{\left\lfloor {n}/{2} \right\rfloor } {\binom n{2k}( - 1)^k } & = (\sqrt 2 )^n \cos\Big(\frac{n\pi}{4}\Big),\\
	\label{eq.hqs0jwi}
	\sum_{k = 0}^{\lfloor n/2\rfloor} \binom{n}{2k}&  = 2^{n-1} ,\\
	\label{eq.k1iddb5}
	\sum_{k = 0}^{\lfloor n/2\rfloor} \binom{n - 1}{2k} & = 2^{n - 2},\quad n\geq2, 
\end{align}
\begin{align}
	\label{eq.rkbacph}
	\sum_{k = 0}^{\left\lfloor {n}/{2} \right\rfloor } {\binom n{2k} ( - 4)^k } & = \sqrt{ 5^n} \cos (n\arctan 2),\\
	\label{eq.cykxsdk}
	\sum_{k = 0}^{\left\lfloor {n}/{2} \right\rfloor } {\binom n{2k} ( - 5)^k }& = \sqrt {6^n} \cos \big(n\arctan\sqrt5\big).
	\end{align}
\end{lemma}
\begin{proof} Setting $x=1$, $z=i$ in~\eqref{eq.b5o7ghm} produces \eqref{eq.wj9bwlz}, while $x=z=1$ gives $2\sum\limits_{k = 0}^{\left\lfloor {n/2} \right\rfloor } \binom{n}{2k}  = 2^n$, from which \eqref{eq.hqs0jwi} and \eqref{eq.k1iddb5} follow. Use of $x=\frac12$, $z=i$, $r=1$, $s=0$ in~\eqref{eq.b5o7ghm} proves~\eqref{eq.rkbacph} while $x=\frac{1}{\sqrt5}$, $z=i$, $r=1$, $s=0$ produces~\eqref{eq.cykxsdk}.
\end{proof}
\begin{theorem}
	If $n$ is a non-negative integer and $r$ and $s$ are integers, then
	\begin{align}\label{eq.jufkvyw}
	10\sum_{k = 0}^{\left\lfloor {n/2} \right\rfloor } {\binom n{2k}F_{k + r} F_{k + s} }  &= L_{2n + r + s}  + ( - 1)^{r + s} L_{n - r - s}- ( - 1)^s 2^{{n}/{2}+1} \cos\Big(\frac{n\pi}{4}\Big)L_{r - s},\\
	\label{eq.jt8fb22}
	2\sum_{k = 0}^{\left\lfloor {n/2} \right\rfloor } {\binom n{2k}L_{k + r} F_{k + s} }  &= F_{2n + r + s}  - ( - 1)^{r + s} F_{n - r - s}- ( - 1)^s 2^{{n}/{2}+1} \cos\Big(\frac{n\pi}{4}\Big)F_{r - s},\\
	\label{eq.xjn6jyy}
	2\sum_{k = 0}^{\left\lfloor {n/2} \right\rfloor } {\binom n{2k}L_{k + r} L_{k + s} }  &= L_{2n + r + s}  + ( - 1)^{r + s} L_{n - r - s}
	 + ( - 1)^s 2^{{n}/{2}+1} \cos\Big(\frac{n\pi}{4}\Big)L_{r - s}.
	\end{align}
\end{theorem}
\begin{proof} From~\eqref{eq.f5dydlo}, we get
	\[
	5\sum_{k = 0}^{\left\lfloor {n/2} \right\rfloor } {\binom n{2k}F_{k + r} F_{k + s} }= \sum_{k = 0}^{\left\lfloor {n/2} \right\rfloor } {\binom n{2k}L_{2k + r + s} }  - ( - 1)^s L_{r - s} \sum_{k = 0}^{\left\lfloor {n/2} \right\rfloor } {( - 1)^k \binom n{2k}}
	\]
	and hence~\eqref{eq.jufkvyw}, upon use of~\eqref{eq.j0dzhsw} and \eqref{eq.wj9bwlz}. The proof of \eqref{eq.jt8fb22},  \eqref{eq.xjn6jyy} is similar.
\end{proof}
\begin{theorem} If $n$ is a positive integer and $r$ and $s$ are any integers, then
	\begin{equation}\label{eq.ul7brs8}
	10\sum_{k = 0}^n {\binom{2n}{2k}4^{k} F_{k + r} F_{k + s} } = L_{6n + r + s}  + 5^n L_{r + s}  - ( - 1)^s 5^n 2L_{r - s} \cos \big(2n\arctan 2\big),
	\end{equation}
	\begin{equation}\label{eq.yfrpnwe}
	\begin{split}
	10\sum_{k = 0}^{n - 1} \binom{2n - 1}{2k}&4^{k} F_{k + r} F_{k + s}  \\
	&\quad= L_{6n + r + s - 3}  - 5^n F_{r + s}  - ( - 1)^s 5^{n - \frac12} 2L_{r - s} \cos \big((2n - 1)\arctan 2\big),
	\end{split}
	\end{equation}
	\begin{equation}
	2\sum_{k = 0}^n {\binom{2n}{2k}4^{k} L_{k + r} F_{k + s} } = F_{6n + r + s}  + 5^n F_{r + s}  - ( - 1)^s 5^n 2F_{r - s} \cos \big(2n\arctan 2\big),\nonumber
		\end{equation}
	\begin{equation}
	\begin{split}
	2\sum_{k = 0}^{n - 1} \binom{2n - 1}{2k+1}&4^{k} L_{k + r} F_{k + s} \\
	&\quad= F_{6n + r + s - 3}  - 5^{n - 1} L_{r + s}  - ( - 1)^s 5^{n - \frac12} 2F_{r - s} \cos \big((2n - 1)\arctan 2\big),\nonumber
	\end{split}
	\end{equation}
	\begin{equation}
	2\sum_{k = 0}^n {\binom{2n}{2k}4^{k} L_{k + r} L_{k + s} } = L_{6n + r + s}  + 5^n L_{r + s}  + ( - 1)^s 5^n 2L_{r - s} \cos\,(2n\arctan 2),\nonumber
		\end{equation}
	\begin{equation}
	\begin{split}
	2\sum_{k = 0}^{n - 1} \binom{2n - 1}{2k}&4^{k} L_{k + r} L_{k + s}  \\
	&\qquad= L_{6n + r + s - 3}  - 5^n F_{r + s}  + ( - 1)^s 5^{n - \frac12} 2L_{r - s} \cos\big ((2n - 1)\arctan 2\big).\nonumber
	\end{split}
	\end{equation}
\end{theorem}
\begin{proof}
	Using~\eqref{eq.f5dydlo} we have
	\[
	10\sum_{k = 0}^{\left\lfloor {n/2} \right\rfloor } {\binom n{2k}4^{k} F_{k + r} F_{k + s} }= 2\sum_{k = 0}^{\left\lfloor {n/2} \right\rfloor } {\binom n{2k}4^{k} L_{2k + r + s} }  - ( - 1)^s 2L_{r - s} \sum_{k = 0}^{\left\lfloor {n/2} \right\rfloor } {( - 4)^k \binom n{2k}},
	\]
	from which \eqref{eq.ul7brs8} and \eqref{eq.yfrpnwe} now follow on account of \eqref{eq.bhdgtf0} and \eqref{eq.rkbacph}.
	
	The remaining identities can be similarly proved, using \eqref{eq.eh9zyvi}, \eqref{eq.z9rhta5} and the identities stated in Theorem~\ref{thm.gybpu6k}.
\end{proof}
\begin{theorem} If $n$ is a positive integer and $r$, $s$ and $j$ are any integers, then
	\begin{align}\label{eq.o7fv5yt}
	10\sum_{k = 0}^n {\binom{2n}{2k}F_{j(2k + r)} F_{j(2k + s)} } &= \left( {L_j^{2n}  + 5^n F_j^{2n} } \right)L_{j(2n + r + s)}  - ( - 1)^{js} 4^{n} L_{j(r - s)},\nonumber\\
	2\sum_{k = 0}^n {\binom{2n}{2k}L_{j(2k + r)} F_{j(2k + s)} }&= (L_j^{2n}  + 5^n F_j^{2n}) F_{j(2n + r + s)}  - ( - 1)^{js} 4^{n} F_{j(r - s)},\\
	2\sum_{k = 0}^n {\binom{2n}{2k}L_{j(2k + r)} L_{j(2k + s)} } &= (L_j^{2n}  + 5^n F_j^{2n}) L_{j(2n + r + s)}  + ( - 1)^{js} 4^{n} L_{j(r - s)},\nonumber
	\end{align}
	\vspace{-16pt}
	\begin{align*}
	10\sum_{k = 0}^{n - 1} &{\binom{2n - 1}{2k} F_{j(2k + r)} F_{j(2k + s)} } \\
	&= ( - 1)^j L_j^{2n - 1} L_{j(2n + r + s - 1)} - ( - 1)^j 5^n F_j^{2n - 1} F_{j(2n + r + s - 1)}- ( - 1)^{js} 2^{2n - 1} L_{j(r - s)}, \\
	2\sum_{k = 0}^{n - 1} &{\binom{2n - 1}{2k}L_{j(2k + r)} F_{j(2k + s)} }\\
	&= ( - 1)^j L_j^{2n - 1} F_{j(2n + r + s - 1)} - ( - 1)^j 5^{n - 1} F_j^{2n - 1} L_{j(2n + r + s - 1)} - ( - 1)^{js} 2^{2n - 1} F_{j(r - s)}, 
	\end{align*}
	\begin{equation*}
	\begin{split}
	2\sum_{k = 0}^{n - 1} &\binom{2n - 1}{2k}L_{j(2k + r)}  L_{j(2k + s)}\\
	  &= ( - 1)^j L_j^{2n - 1} L_{j(2n + r + s - 1)} - ( - 1)^j 5^n F_j^{2n - 1} F_{j(2n + r + s - 1)} + ( - 1)^{js} 2^{2n - 1} L_{j(r - s)}.
	\end{split}
	\end{equation*}
\end{theorem}
\begin{proof}
	We prove~\eqref{eq.o7fv5yt}. The proof of each of the remaining identities is similar and requires the identities given in Theorem~\ref{thm.rz6m834}. In~\eqref{eq.f5dydlo} write $2kj$ for $k$, $rj$ for $r$ and $sj$ for $s$ to obtain
	$$5F_{j(2k + r)} F_{j(2k + s)}  = L_{j(4k + r + s)}  - ( - 1)^{js} L_{j(r - s)} .$$
	Thus,
	\[
	10\sum_{k = 0}^n {\binom{2n}{2k}F_{j(2k + r)} F_{j(2k + s)} }	= 2\sum_{k = 0}^n {\binom{2n}{2k}L_{j(4k + r + s)} }  - ( - 1)^{js} 2L_{j(r - s)}\, \sum_{k = 0}^n \binom{2n}{2k},
	\]
	from which \eqref{eq.o7fv5yt} follows after using~\eqref{eq.hrs065m} and \eqref{eq.hqs0jwi}.
\end{proof}
\begin{theorem}
	If $n$ is a positive integer and $r$ and $s$ are any integers, then
	\begin{align}\label{eq.x9c3jzd}
		5\sum_{k = 0}^{\left\lfloor {n/2} \right\rfloor } {\binom n{2k}F_{3k + r} F_{3k + s} } & = 2^{n - 1} \big( {L_{2n + r + s}  + ( - 1)^n L_{n + r + s} } \big) - ( - 1)^s \sqrt{2^n} \cos \Big(\frac{n\pi}{4}\Big) L_{r - s},\\
	\sum_{k = 0}^{\left\lfloor {n/2} \right\rfloor } {\binom n{2k}L_{3k + r} F_{3k + s} }  & = 2^{n - 1} \bigl( {F_{2n + r + s}  + ( - 1)^n F_{n + r + s} } \bigr) - ( - 1)^s \sqrt{2^n} \cos\Big(\frac{n\pi}{4}\Big)F_{r - s},\nonumber\\
	\sum_{k = 0}^{\left\lfloor {n/2} \right\rfloor } {\binom n{2k}L_{3k + r} L_{3k + s} }  &= 2^{n - 1} \bigl( {L_{2n + r + s}  + ( - 1)^n L_{n + r + s} } \bigr) + ( - 1)^s  \sqrt{2^n} \cos\Big(\frac{n\pi}{4}\Big)L_{r - s}.\nonumber
		\end{align}
\end{theorem}
\begin{proof} We prove only \eqref{eq.x9c3jzd}. Using~\eqref{eq.f5dydlo}, write
	\begin{equation}\label{eq.k6ln87h}
	5F_{3k + r} F_{3k + s}  = L_{6k + r + s}  - ( - 1)^{k + s} L_{r - s} .
	\end{equation}
	Thus,
	\[
	5\sum_{k = 0}^{\left\lfloor {n/2} \right\rfloor } {\binom n{2k}F_{3k + r} F_{3k + s} }  = \sum_{k = 0}^{\left\lfloor {n/2} \right\rfloor } {\binom n{2k}L_{6k + r + s} }  - L_{r - s} \sum_{k = 0}^{\left\lfloor {n/2} \right\rfloor } {( - 1)^{k+s} \binom n{2k}} 
	\]
	and hence~\eqref{eq.x9c3jzd}, using \eqref{L_id_new1} and \eqref{eq.wj9bwlz}.
\end{proof}

Using \eqref{eq.rnor47h}, we have the following.
\begin{corollary}
	If $n$ is a positive odd integer and $r$ and $s$ are any integers, then
	\[
	\sum_{k = 0}^n {\binom{2n}{2k}F_{3k + r} F_{3k + s} }  = 2^{2n - 1}F_n F_{3n + r + s}.
	\]
\end{corollary}
\begin{theorem}
	If $n$ is a non-negative integer and $r$ and $s$ are any integers, then
	\begin{equation}\label{eq.lopcp6m}
	\begin{split}
	\sum_{k = 0}^{\left\lfloor {n/2} \right\rfloor } {\binom n{2k}5^{k-1} F_{3k + r} F_{3k + s} }  &= 2^{n-1}\big(2^{n} L_{2n + r + s}  + ( - 1)^n L_{3n + r + s}\big) \\
	&\quad - ( - 1)^s  \sqrt {6^n} \cos \big(n\arctan\sqrt 5\big)L_{r - s},
	\end{split}
	\end{equation}
	\begin{equation*}
	\begin{split}
	\sum_{k = 0}^{\left\lfloor {n/2} \right\rfloor } {\binom n{2k}5^k L_{3k + r} F_{3k + s} }  &= 2^{n - 1}\big(2^n F_{2n + r + s}  + ( - 1)^n F_{3n + r + s}\big) \\
	&\quad - ( - 1)^s \sqrt {6^n} \cos \big(n\arctan\sqrt 5\big) F_{r - s},
	\end{split}
	\end{equation*}
	\begin{equation*}
	\begin{split}
	\sum_{k = 0}^{\left\lfloor {n/2} \right\rfloor } {\binom n{2k}5^k L_{3k + r} L_{3k + s} }  &= 2^{n - 1}\big(2^n L_{2n + r + s}  + ( - 1)^n L_{3n + r + s}\big) \\
	&\quad + ( - 1)^s \sqrt{6^n} \cos \big(n\arctan\sqrt 5\big) L_{r - s}.
	\end{split}
	\end{equation*}
\end{theorem}
\begin{proof} We prove \eqref{eq.lopcp6m}. Using~\eqref{eq.k6ln87h} we have
	\[
	\sum_{k = 0}^{\left\lfloor {n}/{2} \right\rfloor } \binom n{2k}5^{k-1} F_{3k + r} F_{3k + s} = \sum_{k = 0}^{\left\lfloor {n}/{2} \right\rfloor }{\binom n{2k}5^k L_{6k + r + s} }  - ( - 1)^s L_{r - s} \sum_{k = 0}^{\left\lfloor {n}/{2} \right\rfloor } {\binom n{2k} ( - 5)^k }
	\]
	and hence~\eqref{eq.lopcp6m} using~\eqref{eq.a02pcu8} and~\eqref{eq.cykxsdk}.
\end{proof}

\section{Cubic binomial summation  identities}

In this section we derive some binomial identities involving the product of three Fibonacci and/or Lucas numbers. The identities stated in the next lemma are needed for this purpose.
\begin{lemma}
	For any integers $k$, $r$, $s$ and $t$,
	\begin{align}
	5F_{k + r} F_{k + s} F_{k + t}  &= F_{3k + r + s + t}  - ( - 1)^{k + r} F_{k + s + t - r}  - ( - 1)^{k + t} L_{s - t} F_{k + r},\label{eq.urjg5k6}\\
	5L_{k + r} F_{k + s} F_{k + t} &= L_{3k + r + s + t}  + ( - 1)^{k + r} L_{k + s + t - r}  - ( - 1)^{k + t} L_{s - t} L_{k + r}, \nonumber\\
	L_{k + r} L_{k + s} F_{k + t} & = F_{3k + r + s + t}  + ( - 1)^{k + r} F_{k + s + t - r}  - ( - 1)^{k + t} F_{s - t} L_{k + r}, \nonumber\\
	L_{k + r} L_{k + s} L_{k + t} & = L_{3k + r + s + t}  + ( - 1)^{k + r} L_{k + s + t - r}  + ( - 1)^{k + t} L_{s - t} L_{k + r}\label{eq.hwl0m82}. 
	\end{align}
\end{lemma}
\begin{proof} These can be derived from the identities stated in Lemma~\ref{lem.lk77qra}. Identities~\eqref{eq.urjg5k6} and~\eqref{eq.hwl0m82} are also given in~\cite{kronenburg2}.
\end{proof}
\begin{theorem} If $n$ is a non-negative integer and $r$, $s$ and $t$ are any integers, then
	\begin{equation}\label{eq.sp3km0n}
	\begin{split}
	10&\sum_{k = 0}^{\left\lfloor {n}/{2} \right\rfloor } \binom n{2k}F_{2k + r} F_{2k + s} F_{2k + t}  = 2^n \big( {F_{2n + s + r + t}  + ( - 1)^n F_{n + s + r + t} } \bigr)\\
	&\quad\quad - ( - 1)^r \bigl( {F_{2n + s + t - r}  - ( - 1)^{s + t - r} F_{n - s - t + r} } \bigr) - ( - 1)^t L_{s - t} \bigl( {F_{2n + r}  - ( - 1)^r F_{n - r} } \bigr),
	\end{split}
	\end{equation}
	\begin{equation*}
	\begin{split}
	10&\sum_{k = 0}^{\left\lfloor {n}/{2} \right\rfloor } \binom n{2k}L_{2k + r} F_{2k + s} F_{2k + t}  = 2^n \bigl( {L_{2n + s + r + t}  + ( - 1)^n L_{n + s + r + t} } \bigr)\\
	&\quad\quad + ( - 1)^r \bigl( {L_{2n + s + t - r}  + ( - 1)^{s + t - r} L_{n - s - t + r} } \bigr) - ( - 1)^t L_{s - t} \bigl( {L_{2n + r}  + ( - 1)^r L_{n - r} } \bigr),
	\end{split}
	\end{equation*}
	\begin{equation*}
	\begin{split}
	2&\sum_{k = 0}^{\left\lfloor {n}/{2} \right\rfloor } \binom n{2k}L_{2k + r} L_{2k + s} F_{2k + t}  = 2^n \bigl( {F_{2n + s + r + t}  + ( - 1)^n F_{n + s + r + t} } \bigr)\\
	&\quad\quad + ( - 1)^r \bigl( {F_{2n + s + t - r}  - ( - 1)^{s + t - r} F_{n - s - t + r} } \bigr) - ( - 1)^t F_{s - t} \bigl( {L_{2n + r}  + ( - 1)^r L_{n - r} } \bigr),
	\end{split}
	\end{equation*}
	\begin{equation*}
	\begin{split}
	2&\sum_{k = 0}^{\left\lfloor {n}/{2} \right\rfloor } \binom n{2k}L_{2k + r} L_{2k + s} L_{2k + t}   = 2^n \bigl( {L_{2n + s + r + t}  + ( - 1)^n L_{n + s + r + t} } \bigr)\\
	&\quad\quad + ( - 1)^r \bigl( {L_{2n + s + t - r}  + ( - 1)^{s + t - r} L_{n - s - t + r} } \bigr)+ ( - 1)^t L_{s - t} \bigl( {L_{2n + r}  + ( - 1)^r L_{n - r} } \bigr).
	\end{split}
	\end{equation*}
\end{theorem}
\begin{proof} Write $2k$ for $k$ in~\eqref{eq.urjg5k6} and sum to obtain
	\[
	\begin{split}
	5\sum_{k = 0}^{\left\lfloor {n}/{2} \right\rfloor } \binom n{2k}&F_{2k + r} F_{2k + s} F_{2k + t}  = \sum_{k = 0}^{\left\lfloor {n}/{2} \right\rfloor } {\binom n{2k}F_{6k + r + s + t} } \\
	&- ( - 1)^r \sum_{k = 0}^{\left\lfloor {n}/{2} \right\rfloor } {\binom n{2k}F_{2k + s + t - r} }- ( - 1)^t L_{s - t} \sum_{k = 0}^{\left\lfloor {n}/{2} \right\rfloor } {\binom n{2k}F_{2k + r} },
	\end{split}
	\]
	whence~\eqref{eq.sp3km0n} in view of~\eqref{eq.ffgnfzm} and~\eqref{F_id_new1}.
\end{proof}

\section{Quartic binomial summation  identities}

We conclude our study with the derivation of some quartic binomial Fibonacci identities. The identities stated in the next lemma are required.
\begin{lemma}
	If $k$, $p$, $q$, $r$ and $s$ are any integers, then
	\begin{equation}\label{eq.dh8k87a}
	\begin{split}
	25F_{k + p} F_{k + q} F_{k + r} F_{k + s}&= L_{4k + p + q + r + s}  - ( - 1)^{s + k} L_{2k + p + q + r - s}  - ( - 1)^{r + k} L_{2k + p + q - r + s} \\
	&\quad- ( - 1)^{q + k} L_{p - q} L_{2k + r + s}  + ( - 1)^{r + s} L_{p + q - r - s}  + ( - 1)^{q + s} L_{p - q} L_{r - s}, 
	\end{split}
	\end{equation}
	\begin{equation*}
	\begin{split}
	5L_{k + p} F_{k + q} F_{k + r} F_{k + s}	&= F_{4k + p + q + r + s}  - ( - 1)^{s + k} F_{2k + p + q + r - s}  - ( - 1)^{r + k} F_{2k + p + q - r + s} \\
	&\quad- ( - 1)^{q + k} F_{p - q} L_{2k + r + s}  + ( - 1)^{r + s} F_{p + q - r - s}  + ( - 1)^{q + s} F_{p - q} L_{r - s}, 
	\end{split}
	\end{equation*}
	\begin{equation*}
	\begin{split}
	5L_{k + p} L_{k + q} F_{k + r} F_{k + s} 	&= L_{4k + p + q + r + s}  - ( - 1)^{s + k} L_{2k + p + q + r - s}  - ( - 1)^{r + k} L_{2k + p + q - r + s} \\
	&\quad+ ( - 1)^{q + k} L_{p - q} L_{2k + r + s}  + ( - 1)^{r + s} L_{p + q - r - s}  - ( - 1)^{q + s} L_{p - q} L_{r - s}, 
	\end{split}
	\end{equation*}
	\begin{equation*}
	\begin{split}
	L_{k + p} L_{k + q} L_{k + r} F_{k + s} &= F_{4k + p + q + r + s}  - ( - 1)^{s + k} F_{2k + p + q + r - s}  + ( - 1)^{r + k} F_{2k + p + q - r + s} \\
	&\quad+ ( - 1)^{q + k} L_{p - q} F_{2k + r + s}  - ( - 1)^{r + s} F_{p + q - r - s}  - ( - 1)^{q + s} L_{p - q} F_{r - s}, 
	\end{split}
	\end{equation*}
	\begin{equation}\label{eq.cfrtedwr}
	\begin{split}
	L_{k + p} L_{k + q} L_{k + r} L_{k + s} &= L_{4k + p + q + r + s}  + ( - 1)^{s + k} L_{2k + p + q + r - s}  + ( - 1)^{r + k} L_{2k + p + q - r + s} \\
	&\quad+ ( - 1)^{q + k} L_{p - q} L_{2k + r + s}  + ( - 1)^{r + s} L_{p + q - r - s}  + ( - 1)^{q + s} L_{p - q} L_{r - s}. 
	\end{split}
	\end{equation}
\end{lemma}
\begin{proof} These can be derived from the identities stated in Lemma~\ref{lem.lk77qra}. Identities~\eqref{eq.dh8k87a} and~\eqref{eq.cfrtedwr} are also given in~\cite{kronenburg2}.
\end{proof}
\begin{theorem} If $n$ is a positive integer and $p$, $q$, $r$ and $s$ are any integers, then
	\begin{equation}\label{eq.wqn2nu3}
	\begin{split}
	50\sum_{k = 0}^n \binom{2n}{2k} &F_{k + p} F_{k + q} F_{k + r} F_{k + s}  \\
	&\quad= (5^n  + 1)L_{2n + p + q + r + s}  - ( - 1)^s 2\cdot 5^{\frac{n}{2}} L_{n + p + q + r - s} \cos (2n\arctan \alpha )\\
	&\quad\quad-  2\cdot 5^{\frac{n}{2}} \cos (2n\arctan \alpha ) \big((-1)^r L_{n + p + q - r + s} + ( - 1)^q L_{p - q} L_{n + r + s}\big)\\
	&\quad\quad+ (-1)^s 4^n\big((- 1)^{r} L_{p + q - r - s}  + ( - 1)^{q} L_{p - q} L_{r - s}\big), 
	\end{split}
	\end{equation}
	\begin{equation*}
	\begin{split}
	10\sum_{k = 0}^n \binom{2n}{2k}&L_{k + p} F_{k + q} F_{k + r} F_{k + s}  \\
	&\quad= (5^n  + 1)F_{2n + p + q + r + s}  - ( - 1)^s 2\cdot 5^{{n}/{2}} F_{n + p + q + r - s} \cos (2n\arctan \alpha )\\
	&\quad\quad- 2\cdot 5^{{n}/{2}}\cos (2n\arctan \alpha )\big(( - 1)^r F_{n + p + q - r + s}  + ( - 1)^q F_{p - q} L_{n + r + s}\big)\\
	&\quad\quad + (-1)^s 4^n\big(( - 1)^{r} F_{p + q - r - s}  + ( - 1)^{q} F_{p - q} L_{r - s}\big), 
	\end{split}
	\end{equation*}
	\begin{equation*}
	\begin{split}
	10\sum_{k = 0}^n \binom{2n}{2k} &L_{k + p} L_{k + q} F_{k + r} F_{k + s}  \\
	&\quad= (5^n  + 1)L_{2n + p + q + r + s}  - ( - 1)^s2\cdot 5^{{n}/{2}} L_{n + p + q + r - s} \cos (2n\arctan \alpha )\\
	&\quad\quad- 2\cdot5^{{n}/{2}}\cos (2n\arctan \alpha ) \big(( - 1)^r  L_{n + p + q - r + s}  - ( - 1)^q L_{p - q} L_{n + r + s}\big)\\
	&\quad\quad +(-1)^s4^n\big( ( - 1)^{r} L_{p + q - r - s}  - ( - 1)^{q} L_{p - q} L_{r - s}\big), 
	\end{split}
	\end{equation*}
	\begin{equation*}
	\begin{split}
	2\sum_{k = 0}^n \binom{2n}{2k} &L_{k + p} L_{k + q} L_{k + r} F_{k + s}  \\
	&\quad= (5^n  + 1)F_{2n + p + q + r + s}  - ( - 1)^s 2\cdot 5^{{n}/{2}} F_{n + p + q + r - s} \cos (2n\arctan \alpha )\\
	&\quad\quad+2\cdot 5^{{n}/{2}}\cos (2n\arctan \alpha ) \big( ( - 1)^r F_{n + p + q - r + s} + ( - 1)^q L_{p - q} F_{n + r + s})\\
	&\quad\quad - (-1)^s 4^n \big(( - 1)^{r} F_{p + q - r - s}  + ( - 1)^{q} L_{p - q} F_{r - s}\big), 
	\end{split}
	\end{equation*}
	\begin{equation*}
	\begin{split}
	2\sum_{k = 0}^n \binom{2n}{2k}&L_{k + p} L_{k + q} L_{k + r} L_{k + s}  \\
	&\quad= (5^n  + 1)L_{2n + p + q + r + s}  + ( - 1)^s2\cdot 5^{{n}/{2}} L_{n + p + q + r - s} \cos (2n\arctan \alpha )\\
	&\quad\quad+  2\cdot 5^{{n}/{2}}\cos (2n\arctan \alpha ) \big(( - 1)^r L_{n + p + q - r + s}  + ( - 1)^q L_{p - q} L_{n + r + s}\big)\\
	&\quad\quad + (-1)^s4^n\big(( - 1)^{r} L_{p + q - r - s}  + ( - 1)^{q} L_{p - q} L_{r - s}\big). 
	\end{split}
	\end{equation*}
\end{theorem}
\begin{proof} From \eqref{eq.dh8k87a}, we have
	\[
	\begin{split}
	25\sum_{k = 0}^n \binom{2n}{2k}&F_{k + p} F_{k + q} F_{k + r} F_{k + s} \\ 
	&= \sum_{k = 0}^n {\binom{2n}{2k}L_{4k + p + q + r + s} }  - ( - 1)^s \sum_{k = 0}^n {\binom{2n}{2k}( - 1)^k L_{2k + p + q + r - s} }\\ 
	&\quad - ( - 1)^r \sum_{k = 0}^n {\binom{2n}{2k}( - 1)^k L_{2k + p + q - r + s} }  - ( - 1)^q L_{p - q} \sum_{k = 0}^n {\binom{2n}{2k}( - 1)^k L_{2k + r + s} }\\ 
	&\quad + ( - 1)^{r + s} L_{p + q - r - s} \sum_{k = 0}^n \binom{2n}{2k}  + ( - 1)^{q + s} L_{p - q} L_{r - s} \sum_{k = 0}^n \binom{2n}{2k},
	\end{split}
	\]
	which yields \eqref{eq.wqn2nu3} on account of~\eqref{eq.hrs065m}, \eqref{eq.vfsdrta} and \eqref{eq.hqs0jwi}.
\end{proof}

\end{document}